\documentclass[11pt]{amsart}
\usepackage{amstext,amsfonts,amssymb,amscd,amsbsy,amsmath,verbatim, mathrsfs, fullpage}
\usepackage[alphabetic,abbrev,lite]{amsrefs} 
\usepackage{ifthen,tikz}
\usepackage{color}
\usepackage{amsthm}
\usepackage{latexsym}
\usepackage[all]{xy}
\usepackage{enumerate}
\usepackage{mathtools,accents}
\numberwithin{equation}{section}

\newtheorem{lemma}[equation]{Lemma}

\newtheorem{prop}[equation]{Proposition}
\newtheorem{cor}[equation]{Corollary}

\newtheorem{thm}[equation]{Theorem}

\theoremstyle{definition}
\newtheorem{defn}[equation]{Definition}
\newtheorem{example}[equation]{Example}
\newtheorem{notation}[equation]{Notation}

\newtheorem{algorithm}[equation]{Algorithm}

\theoremstyle{remark}
\newtheorem{remark}[equation]{Remark}

\newtheorem*{claim*}{Claim}
\newtheorem*{case*}{Case}

\newcommand{\Tate}{{\mathbf{T}}}

\newcommand{\cC}{\mathcal{C}}

\newcommand{\m}{\mathfrak m}
\newcommand{\PP}{\mathbb P}
\newcommand{\PPv}{\mathbb P_{\on{var}}}

\newcommand{\bA}{\mathbb A}
\newcommand{\A}{\bA}

\newcommand{\ZZ}{\mathbb Z}

\newcommand{\im}{\operatorname{im}}

\newcommand{\Proj}{\operatorname{Proj}}

\newcommand{\Tor}{\operatorname{Tor}}

\newcommand{\Hom}{\operatorname{Hom}} 

\newcommand{\cO}{{\mathcal O}}
\newcommand{\kk}{{\bf k}}
\newcommand{\rank}{\operatorname{rank}}

\newcommand{\F}{\FF}

\newcommand{\defi}[1]{\textsf{#1}} 

\newcommand{\beq}{\begin{displaymath}}
\newcommand{\eeq}{\end{displaymath}}

\def\reg{\operatorname{reg}}

\def\nc{\newcommand}
\def\on{\operatorname}
\nc{\Q}{\mathbb{Q}}
\nc{\RR}{\mathbf{R}}
\nc{\LL}{\mathbf{L}}
\nc{\xra}{\xrightarrow}
\nc{\xla}{\xleftarrow}
\def\a{\mathbf{a}}
\def\om{\omega}

\def\DM{\operatorname{DM}}

\def\th{\on{th}}
\def\F{\mathcal{F}}
\def\coker{\on{coker}}

\nc{\into}{\hookrightarrow}
\nc{\onto}{\twoheadrightarrow}
\nc{\OO}{\mathcal{O}}
\nc{\Z}{\mathbb{Z}}
\nc{\cA}{\mathcal{A}}
\nc{\w}{\widehat}
\nc{\End}{\on{End}}
\nc{\res}{\frac{1}{x_0x_1}}
\nc{\tF}{\widetilde{F}}
\nc{\tG}{\widetilde{G}}
\nc{\tf}{\widetilde{f}}
\nc{\Com}{\on{Com}}

\nc{\G}{\mathbb{G}}
\nc{\cG}{\mathcal{G}}
\nc{\cE}{\mathcal E}
\nc{\cF}{\mathcal F}
\nc{\cR}{\mathcal R}
\nc{\cD}{\mathcal D}
\nc{\cB}{\mathcal B}
\nc{\cT}{\mathcal T}
\nc{\cL}{\mathcal L}

\nc{\bM}{\mathbf M}
\nc{\bN}{\mathbf N}
\nc{\U}{\mathbf U}
\nc{\BM}{\mathbf B \mathbf M}
\nc{\Dsg}{\on{D}_{\on{sg}}}
\nc{\fC}{\mathcal{C}}
\nc{\fG}{\mathcal{G}}
\nc{\N}{\mathbb{N}}

\nc{\del}{\partial}
\nc{\cone}{\on{cone}}
\nc{\D}{\on{D}_{\on{diff}}}
\nc{\DMb}{\on{D}^b_{\DM}}
\nc{\Db}{\on{D}^{\on{b}}}
\nc{\Kb}{\on{K}^{\on{b}}}
\nc{\fm}{\mathfrak{m}}
\nc{\Flag}{\on{Flag}}
\nc{\DMmin}{\DM_{\on{min}}}
\nc{\Ddiff}{\on{D}_{\on{diff}}}
\nc{\Dbdiff}{\on{D}^\on{b}_{\on{diff}}}
\nc{\wO}{\widehat{\OO}}
\nc{\wT}{\widehat{T}}
\nc{\from}{\leftarrow}
\nc{\wLL}{\widetilde{\LL}}
\nc{\augCech}{\widetilde{\cC}}
\nc{\Fold}{\on{Fold}}
\nc{\Ext}{\on{Ext}}
\nc{\RHom}{\on{RHom}}
\nc{\FF}{\mathbf{F}}
\nc{\Comper}{\Com_{\on{per}}}
\nc{\Unfold}{\on{Unfold}}
\nc{\intHom}{\underline{\Hom}}
\nc{\Ex}{\on{Ex}}
\nc{\tg}{\widetilde{g}}

\nc{\B}{\mathcal{B}}
\nc{\K}{\mathcal{K}}
\nc{\kos}{\on{Kos}}
\nc{\Perf}{\on{Perf}}
\nc{\tR}{\widetilde{\cR}}
\nc{\X}{\mathcal{X}}
\nc{\Cl}{\on{Cl}}
\nc{\fU}{\mathcal{U}}
\nc{\bU}{\mathbf U}
\def\co{\colon}
\def\ce{\coloneqq}
\nc{\st}{\on{st}}

\nc{\coh}{\on{coh}}

\def\D{\mathcal{D}}

\nc{\tU}{\U}
\nc{\bC}{\mathbf{C}}
\nc{\aux}{\on{aux}}

\def\phi{\varphi}

\def\T{\mathbf{T}}
\nc\Dsing{\on{D}^{\on{sing}}}

\def\k{\mathbf{k}}
\def\e{\epsilon}
\def\epsilon{\varepsilon}
\def\w{\mathbf w}

\title{Computing weighted sheaf cohomology using noncommutative differential modules}
\author{Michael K. Brown}
\address{Department of Mathematics and Statistics, Auburn University, Auburn, AL}
\email{\texttt{mkb0096@auburn.edu}}

\newcommand{\Manoa}{M\=anoa}
\newcommand{\Hawaii}{Hawai\kern.05em`\kern.05em\relax i}

\author{Daniel Erman}
\address{Department of Mathematics, University of \Hawaii \ at \Manoa, Honolulu, HI}
\email{\texttt{erman@hawaii.edu}}

\def\MR#1{}

  \thanks{The authors were supported by NSF grants
DMS-2200469 and DMS-2302373, respectively}
\date{\today}

\begin{document}
\begin{abstract}
We describe a novel method for computing sheaf cohomology over weighted projective  spaces and stacks using exterior algebra and differential module techniques, generalizing an algorithm due to Eisenbud-Fl\o ystad-Schreyer over projective space. 
\end{abstract}
\maketitle
\section{Introduction}

We develop a novel algorithm for computing sheaf cohomology over weighted projective stacks, and weighted projective spaces, using exterior algebra and differential module techniques.  The specific computational goal is the following: given a coherent sheaf $\cF$ on a weighted projective stack or space, we want to compute all sheaf cohomology groups of $\cF(j)$ for all $j$ in some specified finite range.
Our approach generalizes work of Eisenbud-Fl\o ystad-Schreyer~\cite{EFS}, who utilized Koszul duality and Tate resolutions to transform sheaf cohomology calculations on $\PP^n$ from a computation over the polynomial ring $S$ to a computation over the dual exterior algebra $E$.

While our method is also based upon Tate resolutions, it requires a number of fundamentally new methods.  
As a quick review of the Eisenbud-Fl\o ystad-Schreyer method: they start with a graded module $M=\bigoplus_{e\in \ZZ} M_e$ over a standard graded polynomial ring $S = k[x_0, \dots, x_n]$.  They truncate it to $M_{\geq d}$, where $d$ is greater than $\reg(M)$, the Castelnuovo-Mumford regularity of $M$; then they apply the Koszul duality functor to obtain a free complex $\RR(M_{\geq d})$ of modules  over the exterior algebra $E$ of the form
\[
\RR(M_{\geq d}) = \left[\Hom_\k(E, M_d)\to \Hom_\k(E, M_{d+1})\to \Hom_\k(E, M_{d+2})\to \cdots \right].
\]
Using facts about the truncation $M_{\geq d}$, they prove that $\RR(M_{\geq d})$ is a coresolution of an $E$-module~$P$. They then compute a minimal free resolution of $P$ and utilize the theory of Tate resolutions on $\PP^n$ to relate the Betti numbers of $P$ to sheaf cohomology groups of $\widetilde{M}$.

Adapting this method to the weighted projective setting requires several innovations.  First, when we utilize Koszul duality in this setting, we must pass from working with free complexes to working with differential modules.  For instance, if $M$ is a graded module over a nonstandard graded polynomial ring $S=\kk[x_0,x_1,x_2]$ with degrees $1,2,3$, then the terms of $\RR(M)$ still have the form $\Hom_\k(E,M_e)$, but the differential inherits a nonstandard homological grading, i.e. the differential can ``jump'' by $1,2$ or $3$ steps:
\[
\xymatrixcolsep{5mm}
\xymatrix{
\RR(M_{\geq d}) =\Bigg[\Hom_\k(E, M_d)  \ar[r]^-{}\ar@/_-2pc/[rr]^-{}\ar[r]^-{}\ar@/_-3pc/[rrr]^-{}&\Hom_\k(E, M_{d+1})\ar[r]^-{}\ar[r]^-{}\ar@/_2pc/[rr]_-{}\ar[r]^-{}\ar@/_3pc/[rrr]^-{}&\Hom_\k(E, M_{d+2})\ar[r]^-{}\ar@/_-2pc/[rr]^-{} \ar[r] & \Hom_\k(E, M_{d+3})\ar[r]&\cdots \Bigg]
}
\]
Working in the categorical framework of differential modules requires new foundational computational methods, including a method for computing a minimal free resolution.  Second, in the classical case, the fact that $\RR(M_{\geq d})$ was a coresolution boiled down to well-known facts about linear resolutions of $M_{\geq d}$; but in the weighted case, there are many open questions about the resolutions of such truncations.  Third, in the weighted setting it is simply not the case that $\RR(M_{\geq d})$ is quasi-isomorphic to its homology, even for $d\gg 0$; since $\RR(M_{\geq d})$ generally has infinite rank, we thus need to develop a method for producing a finitely generated subquotient $P$ which is quasi-isomorphic to $\RR(M_{\geq d})$.  Finally, if we overcome all of these, then we can compute the minimal free resolution of the differential module $P$ in the category of differential modules, and we can apply the theory of Tate resolutions on weighted projective stacks to relate the Betti numbers of $P$ to sheaf cohomology groups of $\widetilde{M}$.

One reason to pursue this approach is that there are computational advantages to working over an exterior algebra $E$ as opposed to a symmetric algebra $S$, as the number of monomials of degree $d$ is generally much smaller in $E$ than in $S$.  This streamlines the complexity of symbolic computations such as producing minimal free resolutions, as it avoids the doubly exponential explosion that can occur with symbolic computation over a polynomial ring~\cite{hermann,mayr-meyer,bayer-mumford}.  In reference to their Tate resolution method for $\PP^n$, Decker-Eisenbud~\cite{decker-eisenbud} write: ``In many cases this is the fastest known method for computing cohomology.''  As our paper is about the theoretical foundation for this novel computational approach, and not about the efficiency of a specific implementation, we will not attempt to give specific speed comparisons.  In fact, an efficient implementation of our algorithm would rely upon fast algorithms for performing computations with differential modules, including for computing minimal free resolutions.  This theory remains nascent---especially in comparison to the tremendous work on computing minimal free resolutions of modules---but we hope that this paper will provide motivation for further improvements in that direction.

\medskip 
\subsection{Algorithm for weighted projective stacks}  Let us go into a bit more detail on the main idea behind our algorithm. Let $S$ denote the $\Z$-graded ring $\k[x_0, \dots, x_n]$ with $\deg(x_i) = a_i$, where $\deg(x_i) = a_i \ge 1$, i.e. the coordinate ring of the weighted projective stack $\PP(\mathbf{a}):=\PP(a_0, \dots, a_n)$
(see \S\ref{sec:tate} for background on weighted projective stacks and their relationship with weighted projective spaces). The Koszul dual of~$S$ is the exterior algebra $E = \bigwedge_\k(e_0, \dots, e_n)$, and Koszul duality furnishes an equivalence between coherent sheaves on $\PP(\mathbf{a})$ and certain \emph{differential $E$-modules} (see \S\ref{subsec:DM} for the definition of a differential $E$-module). We show in our previous work \cite{BETate} that this equivalence is described explicitly via the theory of toric Tate resolutions~\cite[Theorem 6.1]{BETate}, and we prove that Tate resolutions transform ranks of sheaf cohomology over weighted projective stacks into Betti numbers of free resolutions of differential $E$-modules~\cite[Theorem 3.3]{BETate}. Our algorithm is an application of this result.

To give an indication of how this works in practice, and the technical issues that arise, let us consider an example. Let $\PP(1, 1, 2)$ denote the weighted projective stack with $\Z$-graded coordinate ring $S = \k[x_0, x_1, x_2]$, and let us compute certain sheaf cohomology groups of the coordinate ring $M \ce S/(x_0^4 + x_1^4 + x_2^2)$ of an elliptic curve embedded in $\PP(1, 1, 2)$.   Tate resolutions have infinite rank, and so our goal is to compute a finite rank piece.  

By Serre vanishing,  there are no higher cohomology groups for degrees $d\geq d_0$ for some $d_0$, and so in this case, $H^0(E, \cO_E(d)) = M_d$.  Since the differential $E$-module $\RR(M)$, constructed via Koszul duality, has the form $\bigoplus_d M_d\otimes \omega_E(-d;0)$ (see \S\ref{sec:tate} for a detailed description of the functor $\RR$, and see the Notation section below for the meaning of $\om_E$), the Betti numbers of $\RR(M)$—or equivalently the Betti numbers of $\RR(M_{\geq d_0})$—compute the ranks of the $H^0$ groups in degrees $d\geq d_0$.

For the higher cohomology groups, we need to compute a portion of the Tate resolution, which can be done by taking a minimal free resolution $F$ of $\RR(M_{\ge d_0})$ for any $d_0\gg 0$.  We thus need: (i) an algorithm to compute the minimal free resolution of a finitely generated, differential $E$-module (see Algorithm~\ref{alg:resTwisted}), and (ii) a finitely generated $E$-module $P$ that is quasi-isomorphic to $\RR(M_{\ge d_0})$, so that we can apply this algorithm to produce the desired resolution $F$.  Finally, the specific recipe for matching Betti numbers of $F$ with sheaf cohomology groups is proven in Corollary~\ref{cor:cohomology}, generalizing a result of Eisenbud-Fl\o ystad-Schreyer in the standard graded case~\cite[Theorem 4.1]{EFS} and applying the authors' previous work in~\cite{BETate}. 
We return to the example of the elliptic curve from above throughout the paper, e.g. in Examples~\ref{ex:ellipticCurve}, ~\ref{ex:weightedRegEllipticCurve}, and~\ref{ex:ellipticCurveTate}.

To summarize, the key technical issues involved are the following: 
\begin{enumerate}
    \item  What is the smallest $d_0$ that we can use to truncate?
    \item  Since $\RR(M_{\geq d_0})$ is infinite rank, how small of a piece can we utilize without losing information about its homology?  
    \item  In the standard graded case, for any finitely generated $S$-module $N$ that is $r$-regular, the complex $\RR(N_{\ge r})$ is formal, i.e. quasi-isomorphic to its homology~\cite[Corollary 2.4]{EFS}. But this is simply not the case in the weighted setting (Example~\ref{ex:formal}).  How do we (efficiently) produce an appropriate finitely generated subquotient of $\RR(N_{\ge r})$ that is quasi-isomorphic to the infinite rank differential module $\RR(N_{\geq r})$?
    \item  How many steps of the minimal free resolution algorithm for differential modules do we need to run to obtain all sheaf cohomology groups in a desired range of degrees? 
\end{enumerate}  
Resolving these issues is the fundamental contribution of this paper, as it transforms the theory of Tate resolutions into a workable algorithm.

An overview of our answers to the technical issues are as follows.  For (1): if $H^0_{\mathfrak m}M = 0$, then we can choose $d_0$ to be the (weighted) regularity of $M$, as defined by Benson~\cite{benson} (see Definition~\ref{def:reg}). If $H^0_{\mathfrak m}M \ne 0$, then we need the regularity plus $1$; see Lemma 4.7 for details.  To address (2), we can work with $\RR(M)$ in degrees ranging from $d_0$ to $d_0 + \sigma + 1$, where $\sigma = \sum (\deg(x_i) - 1)$ is the Symonds constant (see Definition~\ref{defn:symonds}); this follows from Lemma~\ref{lem:jumpbyWn}, which provides a sharp bound on the longest ``jumps'' in the differentials of the Tate resolution, for a particular flag structure.  For (3):  our answer, which is a subquotient involving $\sigma+2$ distinct degrees of the initial module $M$, appears in Lemma~\ref{lem:weightedFinitePiece}.  For (4): to go from $\RR(M_{\geq d_0})$ to the sheaf cohomology groups $H^i$ in degree $\alpha$, we need to iterate the resolution algorithm $d_0-\alpha-i$ steps; see Algorithm~\ref{alg:mainweighted}.

\subsection{Algorithm for weighted projective spaces}  To compute sheaf cohomology on a weighted projective space, i.e. on the variety corresponding to the stack, we will bootstrap off of our methods on the stack.  There is a natural coarse moduli space map $\pi\colon \PP(\mathbf{a})\to \PPv(\mathbf{a})$ relating the stack~$\PP(\mathbf{a})$ and the variety $\PPv(\mathbf{a})$.  This map preserves cohomology groups and thus, for a sheaf $\cF$ on the stack, the cohomology groups of $\cF$ and $\pi_*\cF$ are identical.  By Lemma~\ref{lem:lcmAndCoarse}(4), the same holds for twists of $\cF$ by the pullbacks of line bundles from $\PPv(\mathbf{a})$.  The line bundles on the variety $\PPv(\mathbf{a})$ correspond to $\mathcal O_{\PPv(\mathbf{a})}(j)$
 with $j$ a multiple of the least common multiple of $a_0, a_1, \dots, a_n$, which we denote by $\ell \ce \on{lcm}(a_0, \dots, a_n)$.  In particular, if $\cF = \widetilde{M}$, then the Tate resolution methods discussed above compute the sheaf cohomology groups of $\widetilde{M}(j)$ over the variety $\PPv(\mathbf{a})$ provided that $j$ is a multiple of $\ell$.

 What if we twist $\widetilde{M}$ by an integer that is not a multiple of $\ell$?  In this case, we face an additional technical hurdle, because on the weighted projective variety, tensor product may not commute with taking the associated sheaf.  That is, it could be the case on $\PPv(\mathbf{a})$ that $\widetilde{M}(j) \ne \widetilde{M(j)}$.\footnote{We caution the reader that the notations $\widetilde{M}(j)$ and $\widetilde{M(j)}$ can look similar at first glance, but the distinction is an important one over weighted projective spaces.}  (This does not happen on the weighted projective stack.)  So we must consider: 
 \begin{quote}
     {\em Given $M$ and $j$, how do we effectively construct\footnote{Of course, if we had a presentation for the sheaf $\widetilde{M}(j)$, and we could compute global sections, then  by (the proof of) \cite[Theorem 1.1]{mustata}, we could set $M'$ to be $\bigoplus_{e\geq 0} H^0(\PPv(\mathbf{a}), \widetilde{M}(j) \otimes \OO(e))$.  But given a presentation of $M$ and the integer~$j$, we may not necessarily have a way to present the sheaf $\widetilde{M}(j)$.} an $S$-module $M'$ such that $\widetilde{M'} \cong~\widetilde{M}(j)$ as coherent sheaves on the weighted projective space $\PPv(\mathbf{a})$?} 
 \end{quote}
This subtle question appears to have been ignored in prior work.  For instance, given a graded module $M$, the Eisenbud-Musta\c{t}\u{a}-Stillman algorithm~\cite{EMS} computes the sheaf cohomology groups~$H^i(\PPv(\mathbf{a}), \widetilde{M(j)})$, but it simply makes no attempt to compute the sheaf cohomology of~$\widetilde{M}(j)$ when $\widetilde{M}(j) \ne \widetilde{M(j)}$.\footnote{A recent implementation of the Eisenbud-Musta\c{t}\u{a}-Stillman algorithm in {\tt Macaulay2} appears to perform a computation like this for weighted projective spaces, but this is due to a minor bug in the code.} 

We resolve this question in Proposition~\ref{prop:tensorProduct}.  We then show how this result can be combined with the theory of Tate resolutions to compute sheaf cohomology on weighted projective varieties.  The resulting algorithm essentially amounts to computing $\ell$ distinct resolutions, one for each $0\leq j <\ell$, and weaving these together to get the sheaf cohomology groups on the weighted projective variety.  This provides the first algorithm, on a weighted projective space, for computing the cohomology groups of $\mathcal F(j)$ for $j$ in a specified range.
The algorithm on the weighted projective variety amounts to roughly $\ell$ iterations of the corresponding algorithm on the weighted projective stack.

\bigskip

In addition to the aforementioned algorithms for projective space, the theory of Tate resolutions on products of projective spaces was introduced in \cite{EES} and implemented in \verb|Macaulay2|~\cite{M2} in the \verb|TateOnProducts| package~\cite{TateProducts}. For an arbitrary toric variety, the question of effectively computing Tate resolutions remains open~\cite[Question 7.3]{BETate}.

\subsection*{Overview of the paper.} In \S\ref{subsec:DM}, we recall some background on differential $E$-modules, and we describe an algorithm for computing their minimal free resolutions (Algorithm~\ref{alg:resTwisted}). In \S\ref{sec:tate}, we recall some needed facts about weighted projective stacks, weighted Castelnuovo-Mumford regularity, and weighted Tate resolutions. For instance, we prove in \S\ref{sec:tate}  a key result relating the weighted Tate resolution of a sheaf to a sufficiently large truncation of a module representing it. With these technical details in hand, we describe our algorithm for computing sheaf cohomology over weighted projective stacks in \S\ref{sec:alg}, and we adapt our results to weighted projective spaces in \S\ref{sec:variety}.

\subsection*{Acknowledgments} We thank David Eisenbud, Frank-Olaf Schreyer, and Gregory G. Smith for many conversations that inspired and strengthened this work.   We also thank 
Maya Banks, Tara Gomes, Prashanth Sridhar, Eduardo Torres Davila,  Keller VandeBogert, and Sasha Zotine.
Computations in \verb|Macaulay2|~\cite{M2} were essential in the development of this~paper. Finally, we thank the referees for many helpful suggestions.

\subsection*{Notation}
Let $\k$ be a field, $S$ the $\Z$-graded polynomial ring $\k[x_0, \dots, x_n]$ with $a_i \ce \deg(x_i)$ satisfying $1 \le a_0 \le \cdots \le a_n$, and $\m$ the homogeneous maximal ideal of $S$. Set $a \ce \sum_{i = 0}^n a_i$.  Let $E$ be the exterior algebra $\bigwedge_\k(e_0, \dots, e_n)$, equipped with the $\Z^2$-grading given by $\deg(e_i) = (-a_i;-1)$. Given $\Z^2$-graded $E$-modules $N$ and $N'$, we let $\underline{\Hom}_E(N, N')$ denote the $\Z^2$-graded $E$-module of $E$-linear maps from $N$ to $N'$.  Let $\om_E$ denote the canonical module $E(-a; -n-1)$ of $E$.

 \section{Differential $E$-modules}
 \label{subsec:DM}
 
 A \defi{differential $E$-module} $D$ is a $\Z^2$-graded $E$-module $D$ equipped with a degree $(0; -1)$ $E$-linear endomorphism $\del$ such that $\del^2 = 0$. A \defi{morphism} of differential $E$-modules is a degree $(0; 0)$ $E$-linear map that is compatible with differentials. Let $\DM(E)$ denote the category of differential $E$-modules.  Differential modules appear in Cartan-Eilenberg's textbook~\cite{CE}, and they played an important role in the history of rank conjectures related to the Carlsson Conjecture and the Buchsbaum-Eisenbud-Horrocks Conjecture~\cite{carlsson-problem,ABI,devries,IW}; they have also since
been applied in many works on 
commutative algebra~\cite{DM,BEpositivity,  banks-vdb-1, banks-bs-theory}, algebraic geometry~\cite{linear,BETate}, representation theory~\cite{rouquier,wei,xyy,RZ, stai2}, and beyond~\cite{Stai,TH,HJ25}.

 The \defi{homology} of $D \in \DM(E)$ is  $H(D) \ce \ker(\del \co D \to D(0;-1)) / \im(\del \co D(0; 1) \to D)$. A morphism of differential $E$-modules is a \defi{quasi-isomorphism} if it induces an isomorphism on homology. The \defi{mapping cone} of a morphism
$f \co D \to D'$ of differential $E$-modules is the differential $E$-module $\on{cone}(f) \ce D' \oplus D(0, -1)$ equipped with the differential
$\begin{pmatrix}
\del' & f \\
0 & -\del
\end{pmatrix}$.  A morphism of differential $E$-modules is a quasi-isomorphism if and only if its mapping cone is exact. A differential $E$-module~$D$ is called \defi{minimal} if the induced differential on $D \otimes_E \k$ is the zero map.

\begin{remark}
One may alternately think of $E$ as a differential bigraded algebra with homological grading given by $|e_i| = -1$, ``internal grading" given by $\deg(e_i) = -a_i$, and trivial differential. From this perspective, the category $\DM(E)$ is isomorphic to the category of differential bigraded $E$-modules; see \cite[Remark 2.3]{BETate}. As in \cite{BETate}, we work with differential modules rather than dg-modules because the former are more amenable to computation in \verb|Macaulay2|. 
\end{remark}

A differential $E$-module $F$ with differential $\del_F$ is called a \defi{free flag} if it may be equipped with a decomposition $F = \bigoplus_{i \in \Z} F_i$ such that each $F_i$ is free, $F_i = 0$ for $ i \ll 0$, and $\del_F(F_i) \subseteq \bigoplus_{j < i} F_j$ for all $i \ge 0$.  A \defi{minimal free flag resolution} of a differential $E$-module $D$ is a quasi-isomorphism $F \xra{\simeq} D$, where $F$ is a minimal free flag. Minimal free flag resolutions of differential $E$-modules with finitely generated homology are unique up to isomorphism \cite[Theorem B.2]{BETate}.  We introduce the following algorithm for constructing them in our setting.  We remark that flag structures are not unique, and we refer to the choice in this algorithm as the ``twisted flag'' because it interweaves the bidegree on $E$.

\begin{algorithm}[Twisted flag]
\label{alg:resTwisted}
Let $D$ be a differential $E$-module with finitely generated and nonzero homology.  Write $H(D)_{(i)} \ce \bigoplus_{j - d = i} H(D)_{(d;j)}$, and
set $n \ce \min\{ i \text{ : } H(D)_{(i)} \ne 0\}$. 
\begin{enumerate}
\item Let $y_1, \dots, y_t \in D$ be homogeneous cycles that descend to a $\k$-basis of $H(D)_{(n)}$, and let $F_{(n)}$ be a graded free $E$-module of rank $t$ equipped with a basis whose $i^{\th}$ element has the same bidegree as $y_i$.  Define a morphism $\e_n \co F_{(n)} \to D$ of differential $E$-modules (where the differential on $F_{(n)}$ is trivial) that sends the $i^{\th}$ basis element of $F_{(n)}$ to $y_i$.   
\item The map $\e_n$ determines an isomorphism on homology in bidegrees $(d; j)$ with $j - d = n$. If $\cone(\e_n)$ is exact, we are done; $\e_n$ gives a minimal free flag resolution of $D$. Otherwise, apply Step (1) to $\cone(\e_n)$, observing that $H(\cone(\e_n))_{(i)} = 0$ for $i < n+1$. 
\end{enumerate}
Iterating this procedure gives a  free flag resolution $F \ce \bigoplus_{i \ge n} F_{(i)} \xra{\simeq} D$, where each $F_{(i)}$ is finitely generated in degrees of the form $(d; j)$ such that $j - d = i$. Moreover, the differential on $F$ is  minimal, since if $i' < i$, any map $F_{(i)} \to F_{(i')}$ is necessarily minimal.
\end{algorithm}

\section{Tate resolutions on weighted projective stacks}
\label{sec:tate}

\subsection{The nonstandard $\Z$-graded BGG correspondence}

Let $\on{mod}(S)$ denote the category of finitely generated $\Z$-graded $S$-modules. The nonstandard $\Z$-graded version of the BGG correspondence gives a functor
$
\RR \co \on{mod}(S) \to \DM(E)
$
given by
$$
\RR(M) = \bigoplus_{d \in \Z} M_d \otimes_\k \om_E(-d; 0), \quad \text{with differential} \quad \del_\RR(m \otimes f) = \sum_{i = 0}^n x_im\otimes e_if.
$$
The functor $\RR$ induces an equivalence on bounded derived categories; see \cite[\S 2.2]{BETate} for additional background on the nonstandard $\Z$-graded BGG-correspondence (and see also \cite{HHW}). 

\begin{remark}
\label{rem:unfold}
Suppose $S$ is standard graded. One can ``unfold" an object in $\DM(E)$ to obtain a complex of $E$-modules \cite[Definition 2.6]{BETate}. Composing the BGG functor $\RR$ above with this unfolding functor gives the usual BGG functor taking values in complexes~\cite[Remark 2.9(2)]{BETate}.
\end{remark}
\begin{remark}
\label{rem:truncR}
We will be interested in applying the functor $\RR$ to truncations of modules: for any integer $r$, $\RR(M_{\ge r})$ is simply the differential $E$-submodule $\bigoplus_{d \ge r} M_d \otimes_\k \om_E(-d;0)$ of $\RR(M)$. 
\end{remark}

\begin{example}\label{ex:ellipticCurve}
Let us consider the case of $\PP(1,1,2)$, where the Cox ring is $S = k[x_0, x_1, x_2]$ with $\deg(x_0) = 1 = \deg(x_1)$ and $\deg(x_2) = 2$. Let $M= S/(x_0^4+x_1^4+x_2^2)$ be the coordinate ring of a genus $1$ curve in $\PP(1,1,2)$.  The differential module $\RR(M)$ is an infinitely generated free $E$-module whose generators are in bijection with a $\kk$-basis for $M$.  In this example, we have:
\[
\begin{tabular}{c|c|c|c|c|c|c}
    $i$&0&1&2&3&4&$\cdots$ \\ \hline
    $\dim M_i$ & 1 &2&4&6&8&$\cdots$
\end{tabular}
\]
Thus, $\RR(M)$ has the form:
\vskip\baselineskip
\footnotesize
\[
\RR(M) = \left[\xymatrix{
\om_E \ar[r]^-{}\ar@/_2pc/[rr]_-{}& \om_E(-1;0)^{\oplus 2} \ar[r]^-{}\ar@/_-2pc/[rr]^-{}&\om_E(-2;0)^{\oplus 4} \ar[r]^-{}\ar@/_2pc/[rr]_-{} &\om_E(-3;0)^{\oplus 6}\ar[r]^-{} & \cdots
}\right]
\]
\normalsize
The above graphic is an attempt to represent a flag differential module, not a complex.  To be more precise, we choose monomial bases for each $M_i$, e.g. $\{1\}, \{x_0,x_1\}$ and $\{x_0^2,x_0x_1,x_1^2, x_2\}$ for $M_0, M_1$ and $M_2$, respectively.  The differential $\partial$ sends the generator $1 \in \omega_E$ to $e_0\otimes x_0 + e_1\otimes x_1 + e_2\otimes x_2$.  This element lies in $\omega_E(-1;0)^{\oplus 2} \oplus \omega_E(-2; 0)^{\oplus 4}$.  This is indicated in the above graphic by the arrows emanating from $\omega_E$. For any $r \ge 0$, $\RR(M_{\ge r})$ may be depicted by removing all summands of the form $\om_E(-d ; 0)$ for $d < r$ from the picture above.
\end{example}

\subsection{Weighted Castelnuovo-Mumford regularity}\label{subsec:weightedReg}
The following notion of Castelnuovo-Mumford regularity for modules over nonstandard $\Z$-graded polynomial rings was introduced by Benson~\cite{benson}.

\begin{defn}
\label{def:reg}
A graded $S$-module $M$ is called \defi{$r$-regular} if 
$H^i_{\m}(M)_j = 0$
for all $j > r - i$. The \defi{regularity of $M$} is defined to be $\reg(M) \ce \inf\{r \text{ : } \text{$M$ is $r$-regular}\}$.
\end{defn}

If $M$ is a finitely generated graded $S$-module, then \cite[Proposition 1.2]{symonds} implies:
\begin{equation}
\label{eqn:symonds}
\reg(M) = \max\{ j-i \text{ : } \Tor_i^S(M, \k)_j \ne 0\} + n + 1 - a.
\end{equation}
In particular,  $\reg(M) < \infty$. Recalling that $a  = \sum_{i=0}^n a_i = \sum_{i=0}^n \deg(x_i)$, this also leads to the following definition:
\begin{defn}\label{defn:symonds}
    We define the \defi{Symonds constant} for $S$ as $\sigma:= \sum_{i=0}^n (a_i - 1) = a - (n+1)$.
\end{defn}

\begin{example}
\label{ex:weightedRegEllipticCurve}
We have $\reg(S) = -\sigma$, and $\reg(\k) = 0$. If $M$ is the coordinate ring of the elliptic curve from Example~\ref{ex:ellipticCurve}, its $S$-free resolution has the form $0 \from S \from S(-4) \from 0$, and so the formula~\eqref{eqn:symonds} implies that $\reg(M) = 2$. 
\end{example}

\begin{lemma}
\label{rem:trunc}
Let $M$ be a graded $S$-module, and suppose $M$ is $r$-regular.
\begin{enumerate}
\item The module $M_{\ge r+1}$ is also $r$-regular. In particular, $H^0_\m(M_{\ge r+1}) = 0$. 
\item If $H^0_\m(M) = 0$, then $M_{\ge r}$ is $r$-regular. 
\end{enumerate} 
\end{lemma}

\begin{proof}
We have $H^i_{\m}(M/M_{\ge r+1}) = 0$ for $i > 0$, and $H^0_{\m}(M/M_{\ge r+1})_j = (M/M_{\ge r+1})_j = 0$ for $j > r$. The long exact sequence in local cohomology induced by $0 \to M_{\ge r+1} \to M \to M/M_{\ge r+1} \to 0$ therefore implies $H^i_{\m}(M_{\ge r+1}) \cong H^i_{\m}(M)$ for $i > 0$ and $j > r$, and $H^0_{\m}(M_{\ge r+1})_j \cong H^0_{\m}(M)_j = 0$ for~$j > r$. This proves (1). The same argument with $M_{\ge r}$ in place of $M_{\ge r+1}$ proves (2). 
\end{proof}

\subsection{Weighted projective stacks and varieties}
\label{rmks:varietyVsStack}
Let $\a$ be the ordered list $a_0, \dots, a_n$ of the degrees of $x_0, \dots, x_n  \in S$, and let $\PP(\mathbf{a})$ be the weighted projective stack $[\A^{n+1} - \{0\} / \k^*]$ associated to the action of the multiplicative group $\k^*$ on $\A^{n+1}$ given by
$
\lambda \cdot (x_0, \dots, x_n) = (\lambda^{a_0}x_0, \dots, \lambda^{a_n} x_n).
$ 
Weighted projective stacks are examples of toric stacks, in the sense of \cite{BCS} or \cite{GS}; we refer the reader to these two works for background on stacks in the generality we will consider.
Let~$\PPv(\mathbf{a}) \ce \Proj(S)$ denote the associated weighted projective \emph{variety}.  Sheaves on the stack $\PP(\a)$ behave much like sheaves on projective space, while sheaves on the variety $\PPv(\mathbf{a})$ are pathological in many respects. 
Let us now compare the behavior of sheaves on $\PP(\a)$ and $\PPv(\mathbf{a})$ in some detail.

The stack $\PP(\mathbf{a})$ is toric Deligne-Mumford, and so the sheaf $\OO_{\PP(\mathbf{a})}(i)$ on $\PP(\mathbf{a})$ is a line bundle for all $i \in \Z$. By contrast, the variety $\PPv(\mathbf{a})$ is generally singular, and the sheaf  $\cO_{\PPv(\mathbf{a})}(i)$ is not a line bundle unless $ \ell \ce \on{lcm}(a_0, \dots, a_n)$ divides $i$~\cite[Theorem 7.1(c)]{BR}.  This creates various subtleties in working with the variety.  For instance, while the isomorphism $\OO(i) \otimes \OO(j) \cong \OO(i + j)$ holds over~$\PP(\mathbf{a})$, the analogous relation does not always  hold over~$\PPv(\mathbf{a})$~\cite[pp. 134]{BR}. However, we have $\OO(i) \otimes \OO(j) \cong \OO(i + j)$ over $\PPv(\a)$ when $i$ or $j$ is a multiple of $\ell$ \cite[Corollary 4A.5(b)]{BR}.

 The category $\on{coh}(\PP(\mathbf{a}))$ of coherent sheaves on $\PP(\a)$ has a simple algebraic description: it is the abelian quotient of the category $\on{mod}(S)$ of graded $S$-modules by the subcategory of objects with finite dimension over $\k$. Given an object $M$ in $\on{mod}(S)$, we let $\widetilde{M}$ denote the associated sheaf on~$\PP(\mathbf{a})$.   We note that a map of such modules $M\to M'$ induces an isomorphism of coherent sheaves $\widetilde{M}\to \widetilde{M'}$ on $\PP(\mathbf{a})$ if and only if the map is an isomorphism $M_d\to M'_d$ for all $d\gg 0$.
By contrast, the correspondence between sheaves on the variety $\PPv(\mathbf{a})$  and $S$-modules is more complicated.  Every finitely generated, graded $S$-module $M$ induces a corresponding coherent sheaf on the variety, which we also denote by $\widetilde{M}$, but the equivalence relation is more subtle as one must differentiate between degrees that correspond to elements of the Picard group of $\PPv(\mathbf{a})$ and those that correspond to elements of the divisor class group of $\PPv(\mathbf{a})$.  For instance, a map of such modules $M\to M'$ induces an isomorphism of coherent sheaves on $\PPv(\mathbf{a})$ if and only if the map is an isomorphism $M_d\to M'_d$ for all $d$ sufficiently large and divisible by $\ell$~\cite[Exercise 5.3.5]{CLS}.

 Sheaf cohomology computations can be translated between the  stack $\PP(\a)$ and the variety $\PPv(\a)$. More precisely: the pushforward along the coarse moduli space map $\pi \co \PP(\mathbf{a}) \to \PPv(\mathbf{a})$ is exact on coherent sheaves~\cite[Proposition 3.5]{EM12}, and so computing the cohomology of a coherent sheaf $\F$ on~$\PP(\mathbf{a})$ is the same as computing the cohomology of $\pi_*(\F)$ over ~$\PPv(\mathbf{a})$. 

But one needs to be careful when combining sheaf cohomology with twists, as in our primary computational goal of computing the sheaf cohomology groups of all twists within a specified range.  Given a sheaf $\F$ on the stack $\PP(\a)$ and $j \in \Z$, we write $\F(j) \ce \F \otimes \OO(j)$, and similarly over the variety~$\PPv(\a)$. Given a graded $S$-module $M$, we have $\widetilde{M(j)} \cong \widetilde{M}(j)$ over the stack $\PP(\a)$, but this relation does not hold over the variety, as already noted above when $M = S(i)$ for $i \in \Z$. We have $H^i(\PP(\a), \widetilde{M(j)}) \cong H^i(\PPv(\a), \widetilde{M(j)})$ for all $i, j$; one sees this by comparing both to $H^i_\m(M)_j$. But it can happen that $H^i(\PPv(\a), \widetilde{M}(j)) \ncong H^i(\PPv(\a), \widetilde{M(j)})$: see Example~\ref{ex:stackVsVariety}.

\subsection{Tate resolutions on weighted projective stacks} 
Let $\F$ be a coherent sheaf on $\PP(\mathbf{a})$, ~$M$ a finitely generated graded $S$-module such that $\widetilde{M} = \F$, and $r > \reg(M)$. It follows from \cite[Proposition 2.11(a)]{BETate} that, given any finitely generated graded $S$-module $N$, the homology of $\RR(N)$ is finitely generated. We may therefore apply Algorithm~\ref{alg:resTwisted} to construct a minimal free flag resolution $\e \co F \xra{\simeq} \RR(M_{\ge r})$. We let $\Tate(\F)$ denote the exact, free differential $E$-module $\cone(\e)$. The following theorem shows that $\Tate(\F)$ coincides with the \emph{Tate resolution of $\F$}, as defined in \cite[\S 3]{BETate}; in particular, $\Tate(\F)$ is independent of the choice of $r$, and its Betti numbers encode the sheaf cohomology of the twists of $\F$. The Tate resolution functor $\F \mapsto \Tate(\F)$ is a geometric analogue of the weighted  BGG functor $\RR$; see \cite[\S 3]{BETate} for details. 

\begin{thm}
\label{thm:tate}
Let $\F \in \coh(\PP(\mathbf{a}))$, $M$ a finitely generated graded $S$-module such that $\widetilde{M} = \F$, and~$r > \reg(M)$. Let $M'$ denote the $\m$-saturation $\bigoplus_{i \in \Z} H^0(\PP(\mathbf{a}), \F(i))$ of~$M$.
\begin{enumerate}
\item The differential $E$-module $\RR(M')$ has finitely generated homology.
\item Let $\e \co F \xra{\simeq} \RR(M_{\ge r})$ and $\e' \co F' \xra{\simeq}  \RR(M')$ be minimal free flag resolutions, as constructed in Algorithm~\ref{alg:resTwisted}. There is an isomorphism $\cone(\e) \xra{\cong} \cone(\e')$ of differential $E$-modules. 
In particular, our definition of~$\Tate(\F)$ does not depend on $r$.
\item The differential $E$-module $\Tate(\F)$ is minimal.
\item We have $H^j(\PP(\mathbf{a}), \F(i)) \cong \underline{\Hom}_E(\k, \Tate(\F))_{(i, -j)}$ for all $i, j \in \Z$. 
\end{enumerate}
Moreover, if $H^0_\m(M) = 0$, then (2) also holds with $r = \reg(M)$.
\end{thm}

\begin{remark}
Theorem~\ref{thm:tate}(1) does not immediately follow from \cite[Proposition 2.11(a)]{BETate}, since~$M'$ is not necessarily finitely generated. Theorem~\ref{thm:tate}(1) is used implicitly in the statement and proof of \cite[Theorem 3.7]{BETate}, but we provide here a careful proof of this fact; it is an important technical point, as it allows us to take a minimal free resolution of $\RR(M')$. 
\end{remark}

\begin{proof}[Proof of Theorem~\ref{thm:tate}]

Lemma~\ref{rem:trunc}(1) implies $H^0_\m(M_{\ge r}) = 0$, and so there is a short exact sequence 
\begin{equation}\label{eqn:SESlocalCohom}
0 \to M_{\ge r} \to M' \to H^1_{\fm}(M_{\ge r}) \to 0.
\end{equation}
Since the functor $\RR$ is exact, we thus have a short exact sequence
$0 \to \RR(M_{\ge r}) \to \RR(M') \to \RR(H^1_{\fm}(M_{\ge r})) \to 0$, and so we need only show that $\RR(H^1_{\fm}(M_{\ge r}))$ has finitely generated homology. By local duality, we have $H^1_\fm(M_{\ge r}) \cong \Ext^n_S(M_{\ge r}, S(-a))^*$, where $( - )^*$ denotes the $\k$-dual. Since $\Ext^n_S(M_{\ge r}, S(-a))$ is a finitely generated $S$-module, and the functor $\Hom_E( - , E)$ is exact, (1) follows from the straightforward observation that for any $E$-module $N$, $\RR(N^*)$ coincides, up to a grading twist, with $\Hom_E(\RR(N), E)$. Let us now prove~(2). By (1), $\RR(M')$ admits a minimal free flag resolution via Algorithm~\ref{alg:resTwisted} (since $M$ is finitely generated, the similar statement for $M_{\ge r}$ follows from \cite[Proposition 2.11(a)]{BETate}). 

As a free $E$-module (ignoring the differential), we have
\begin{equation}
\label{eqn:H1}
\RR(M') \cong  \RR(H^1_{\fm}(M_{\ge r}))\oplus \RR(M_{\ge r}),
\end{equation}
because \eqref{eqn:SESlocalCohom} splits as $\k$-vector spaces.  Notice that $\RR(M_{\ge r})$ is a sum of free modules of the form $\om_E(-d, 0)$ for $d \ge r$; and $\RR(H^1_{\fm}(M_{\ge r}))$ is a sum of free modules of the form $\om_E(-d, 0)$ for $d < r$, since $M_{\ge r}$ is $r$-regular by Lemma~\ref{rem:trunc}(1). It follows that the differential on $\RR(M')$ may be expressed, via the isomorphism~\eqref{eqn:H1}, as a $2 \times 2$ block matrix
$
\begin{pmatrix} 
\del_{\RR(H^1_{\m}(M_{\ge r}))} & 0 \\ \alpha  & \del_{\RR(M_{\ge r})}
\end{pmatrix}
$
for some degree zero $E$-linear map $\alpha \co \RR(H^1_{\fm}(M_{\ge r}))(0, 1) \to \RR(M_{\ge r})$ such that $\alpha\del_{\RR(H^1_{\m}(M_{\ge r}))} + \del_{\RR(M_{\ge r})}\alpha = 0$. 
Write the augmentation map $\e' \co F' \to \RR(M')$ as the column matrix $\begin{pmatrix} \e'_1 \\ \e'_2\end{pmatrix}$, and let $G$ be the minimal free flag $ F' \oplus \RR(H^1_\fm(M_{\ge r})) $ with differential $\begin{pmatrix} \del_{F'} &  0\\ -\e_1' & -\del_{\RR(H^1_{\m}(M_{\ge r}))}\end{pmatrix}$. The map
$
G \xra{\begin{pmatrix} \e_2' &  \alpha\end{pmatrix}} \RR(M_{\ge r})
$ is a minimal free flag resolution of $\RR(M_{\ge r})$. By the uniqueness of minimal free flag resolutions, we thus have $\cone\begin{pmatrix} \e_2' & \alpha\end{pmatrix} \cong \cone(\e)$. Since $\cone\begin{pmatrix} \e_2' & \alpha\end{pmatrix} \cong \cone(\e')$, we have proven (2).
Parts (3) and (4) follow  from (2) and \cite[Theorem 3.7]{BETate}, and the final statement follows from the above argument combined with Lemma~\ref{rem:trunc}(2).
\end{proof}

Before stating the following corollary of Theorem~\ref{thm:tate}, we fix some notation: given a coherent sheaf $\F$ on $\PP(\mathbf{a})$, we write $h^i(\F) \ce \dim_\kk H^i(\PP(\mathbf{a}), \F)$.  

\begin{cor}
\label{cor:cohomology}
Let $\F \in \coh(\PP(\mathbf{a}))$, $M$ a finitely generated graded $S$-module such that $\widetilde{M} = \F$, $r > \reg(M)$, and $F$ a minimal free flag resolution of $\RR(M_{\ge r})$. We have:
$$
h^i(\F(j)) =
\begin{cases}
\dim_\k M_j, & i = 0 \text{, } j \ge r; \\
0, & i > 0 \text{, } j \ge r; \\
\dim_\k \underline{\Hom}_E(\k, F)_{(j, -i - 1)}, &  j < r. 
\end{cases}
$$
If $H^0_\m(M) = 0$, then the same conclusion holds when $r = \reg(M)$.
\end{cor}

\begin{proof}
We have an exact sequence
$$
0 \to H^0_\fm(M_{\ge r}) \to M_{\ge r} \to \bigoplus_{j \in \Z} H^0(\PP(\mathbf{a}), \F(j)) \to H^1_{\fm}(M_{\ge r}) \to 0
$$
and isomorphisms $H^{i+1}(M_{\ge r})_j \cong H^i(\PP(\mathbf{a}), \F(j))$ for $i \ge 1$. Suppose $j \ge r$. Since $M_{\ge r}$ is $r$-regular (Lemma~\ref{rem:trunc}), we have $H^i(\PP(\mathbf{a}), \F(j)) = 0$ for $i \ge 1$. We also have  $H^0_{\fm}(M_{\ge r}) = 0$, and $H^1_{\fm}(M_{\ge r})_j = 0$ for $j \ge r$. It follows that $H^0(\PP(\mathbf{a}), \F(j)) \cong M_j$ for $j \ge r$. 
Now suppose $j < r$.  
By Theorem~\ref{thm:tate}(4), we have: 
\begin{align*}
h^i(\F(j)) = \dim_\k \underline{\Hom}_E(\k, \Tate(\F))_{(j, -i)} = \dim_\k \underline{\Hom}_E(\k, F(0, -1) \oplus \RR(M_{\ge r}))_{(j, -i)}.
\end{align*}
Since $\underline{\Hom}_E(\k,\RR(M_{\ge r})_{(s,t)} = 0$ for any pair $(s, t) \in \Z^2$ with $s < r$, we conclude:
$$
h^i(\F(j)) = \dim_\k \underline{\Hom}_E(\k, F(0, -1))_{(j, -i)} = \dim_\k \underline{\Hom}_E(\k, F)_{(j, -i - 1)},
$$ 
which completes the proof. 
\end{proof}

\begin{example}
\label{ex:ellipticCurveTate}
Let us continue with Example~\ref{ex:ellipticCurve}, where  $C=V(x_0^4+x_1^4+x_2^2)\subseteq \PP(1,1,2).$ Since the regularity of $M = S/(x_0^4 + x_1^4 + x_2^2)$ is 2 (Example~\ref{ex:weightedRegEllipticCurve}), and $H^0_\m(M) = 0$, Theorem~\ref{thm:tate} implies that the Tate resolution of $\OO_C$ is the mapping cone of the quasi-isomorphism $F \xra{\simeq} \RR(M_{\ge 2})$, where $F$ is the minimal free flag resolution of $\RR(M_{\ge 2})$. 

This minimal free flag resolution can be computed using Algorithm~\ref{alg:resTwisted}, yielding the Tate resolution, which looks as follows:
\begin{equation}\label{eqn:Tateweighted}
\xymatrixcolsep{5mm}
\xymatrix{
\cdots  & \ar@/_2pc/[rr]+<1cm,-2mm>\om_E(2;1)^{\oplus 4} \ar[r] & \omega_E(1;1)^{\oplus 2} \oplus \om_E(0;0)\ar@/_-2pc/[rr]^-{} \ar[r] & \omega_E(0;1) \oplus \omega_E(-1;0)^{\oplus 2}\ar[r] \ar@/_2pc/[rr]^-{}&\omega_E(-2;0)^{\oplus 4} \ar[r]& \cdots
}
\end{equation}
We have grouped together the summands so that each summand of $F_{(i)}$ (using the notation of Algorithm~\ref{alg:resTwisted}) is in the same horizontal ``position".  Arrows linking summands $\omega_E(*;1)$ to summands~$\omega_E(*;0)$ correspond to elements of $E_{(-2, -2)}$ or $E_{(-3, -2)}$. Because of the twisted flag structure, such arrows will only ``jump'' 2 steps in the flag structure.  See, for example, the arrow linking the summands $\omega_E(2;1)$ to the summands $\omega_E(-1;0)$.
\end{example}

See \cite[Examples 3.11--3.13]{BETate} for more examples of weighted Tate resolutions.

\begin{remark}
\label{rem:unfold2}
In the standard graded case, it follows from Theorem~\ref{thm:tate} and Remark~\ref{rem:unfold} that our construction of the Tate resolution may be identified with the construction in \cite[\S 4]{EFS} via the unfolding functor~\cite[Definition 2.6]{BETate}. 
\end{remark}

\begin{example}\label{ex:stackVsVariety}
    In this example, we return to the distinction between weighted projective varieties and stacks, and we illustrate why Tate resolutions are more closely aligned with the stack.  We consider $\PP(1,1,2)$ and the associated variety $\PPv(1,1,2)$.  Let $M=\kk[x_0,x_1,x_2]/(x_0,x_1)\cong \kk[x_2]$.  The corresponding sheaf on $\PP(1,1,2)$ is the structure sheaf $\cO_{p}$ of the (stacky) point $p=[0:0:1]$, and on $\PPv(1,1,2)$ it is the structure sheaf $\mathcal O_q$ of the singular point $q=[0:0:1]$.

We have $\dim_k M_i=1$ when $i$ is nonnegative and even, and $M_i = 0$ otherwise. Thus, the differential $E$-module $\RR(M)$ is
\[
\xymatrixcolsep{5mm}
\xymatrix{
0 & \om_E(0;0) \ar@/_1pc/[rr]^-{} &0& \omega_E(-2;0)\ar@/_-1pc/[rr]^-{}&0& \omega_E(-4;0)\ar@/_1pc/[rr]^-{}&0& \cdots,
}
\]
where all arrows are multiplication by the exterior variable $e_2$.  Resolving $\RR(M)$, we get the Tate resolution, which continues the pattern in the opposite direction:
\begin{footnotesize}
\[
\xymatrixcolsep{5mm}
\xymatrix{
\cdots\ar@/_-1pc/[rr]^-{}  & 0 & \om_E(4;0) \ar@/_1pc/[rr]^-{}& 0 & \om_E(2;0) \ar@/_-1pc/[rr]^-{} &0 & \om_E(0;0) \ar@/_1pc/[rr]^-{} &0& \omega_E(-2;0)\ar@/_-1pc/[rr]^-{}&0& \omega_E(-4;0)\ar@/_1pc/[rr]^-{}&0& \cdots
}.
\]
\end{footnotesize}

\noindent Since $H^0(\PP(1,1,2), \mathcal O_p(i))=1$ if $i$ is even and $0$ if $i$ is odd, the Tate resolution computes the sheaf cohomology of the structure sheaf of the stacky point $\mathcal O_p$. 

For the singular point $q$ on $\PPv(1,1,2)$, the situation is a bit more delicate. The reason for the distinction is that, on $\PPv(1,1,2)$, the cohomology of the sheaves $\widetilde{M(1)}$ and $\widetilde{M} \otimes \cO_{\PPv(1,1,2)}(1)$ do not coincide.  Indeed, we have $\widetilde{M(1)} =0$ on $\PPv(1,1,2)$~\cite[Exercise 5.3.5]{CLS}, and so of course it has no cohomology.  
On the other hand, since tensor product of sheaves preserves the stalks, we have an isomorphism $\widetilde{M}(1)_q \cong \widetilde{M}_q \otimes  \cO_{\PPv(1,1,2)}(1)_q$.  Writing $Q=\langle x_0,x_1\rangle$ for the prime ideal in $S$ associated to the point $q$, we have:
\[
\widetilde{M}_q \otimes  \cO_{\PPv(1,1,2)}(1)_q \cong \left( M_Q\right)_0 \otimes \left( S_Q\right)_1 \cong \kk^2.
\]
We conclude that $H^0(\PPv(1,1,2), \widetilde{M(1)}) = 0 \ne \k^2 = H^0(\PPv(1,1,2), \widetilde{M}(1))$. 

As noted in \S\ref{rmks:varietyVsStack}, we have $H^i(\PPv(\a), \widetilde{M(j)}) \cong H^i(\PP(\a), \widetilde{M}(j))$, and so a Tate resolution can be used to compute $H^i(\PPv(\a), \widetilde{M(j)})$ for all $i, j$. With a more nuanced approach, it is also possible to compute the cohomology table $H^i(\PPv(\a), \widetilde{M}(j))$ using (finitely many) Tate resolutions: see \S\ref{sec:variety}.
\end{example}

\section{The algorithm}
\label{sec:alg}
Given a graded $S$-module $M$, we have $\rank_E \RR(M) = \dim_\k M$, which is typically infinite. Thus, for computational purposes, we can only work with subquotients of~$\RR(M)$. To transform Theorem~\ref{cor:cohomology} into a workable algorithm for computing sheaf cohomology over $\PP(\mathbf{a})$, we therefore need a method for building the Tate resolution of $\widetilde{M}$ out of a finitely generated portion of $\RR(M)$. This is the content of Corollary~\ref{cor:Ndel}, which is the final ingredient we need for our sheaf cohomology algorithm over weighted projective stacks (Algorithm~\ref{alg:mainweighted}).

Given a coherent sheaf $\F$ on $\PP(\mathbf{a})$, we define:
\begin{equation}
\label{eqn:filtration}
\Tate(\F)_\ell \ce \bigoplus_{-i - j = \ell} H^i(\PP(\mathbf{a}), \F(j)) \otimes_\kk \om_E(-j; i).
\end{equation}
By Theorem~\ref{thm:tate}(4), the Tate resolution $\Tate(\F)$ is isomorphic, as a graded $E$-module, to  $\bigoplus_{\ell \in \Z} \T(\F)_{\ell}$. For the rest of the paper, we identify the underlying $E$-module of $\T(\F)$ with this direct sum. 

\begin{remark}
In the standard graded case, it follows from \cite[Theorem 4.1]{EFS} and Remark~\ref{rem:unfold2} that $\T(\F)_\ell$ may be identified, via the unfolding functor~\cite[Definition 2.6]{BETate},  with the $\ell^{\th}$ homological degree term of the Tate resolution as defined in \cite[\S 4]{EFS}. 
\end{remark}

Since the differential $E$-module $\RR(M_{\ge r})$ has infinite rank as an $E$-module, we cannot view it in \verb|Macaulay2| in its entirety. In order to allow for explicit  computation, we therefore need a finitely generated differential $E$-module to which $\RR(M_{\ge r})$ is quasi-isomorphic. In the standard graded case, this is achieved as follows.  First, we have that the Tate resolution is a free complex, and thus the differential always shifts by one in homological degree.  Second, in the standard graded case,
if $r \ge \reg(M)$ then $\RR(M_{\ge r})$ is formal, i.e. it is quasi-isomorphic to its homology~\cite[Corollary 2.4]{EFS}. 
Taken together, these imply that one can compute the Tate resolution by working with just two graded pieces of $M$ and resolving the kernel of 
\[
M_r\otimes \omega_E(-r;0) \to M_{r+1}\otimes \omega_E(-r-1;0).
\]

In the weighted case, we have two obstacles.  Instead of a complex, the best we can hope for is that $\RR(M)$ has a flag structure.  So to bound the portion of $\RR(M)$ we need to consider, we need an upper bound on how much the differential can shift the flag index.  This is the content of Lemma~\ref{lem:jumpbyWn}.  Second, in the weighted case, $\RR(M_{\ge r})$  generally fails to be formal, even if $r\gg 0$.  

\begin{example}\label{ex:formal}
Suppose $S = \k[x_0, x_1]$, where $\deg(x_0) = 1$, and $\deg(x_1) = 2$. If $M = S / (x_0^2 - x_1)$, then it follows from the multigraded Reciprocity Theorem~\cite[Theorem 2.15]{BETate} that $\RR(M_{\ge r})$ is not formal for any $r$. Indeed, since $M_{\ge r}$ does not have a strongly linear free resolution in the sense of \cite[Definition 1.2]{linear} for any $r$, it cannot be quasi-isomorphic to~$\LL(N)$ for some $E$-module~$N$, where $\LL$ is the BGG functor defined in \cite[\S 2.2]{BETate}.
\end{example}
When $\RR(M_{\ge r})$ fails to be formal, as in the above example, we thus need an alternate method for producing a finitely generated differential module that is quasi-isomorphic to $\RR(M_{\geq r})$: this is given by Lemma~\ref{lem:weightedFinitePiece}. Let us now begin establishing these technical facts.

\begin{lemma}\label{lem:jumpbyWn}
    The differential on $\Tate(\F)$ sends $\Tate(\F)_{\ell}$ to $\Tate(\F)_{\ell - 1} \oplus \cdots \oplus \Tate(\F)_{\ell - \sigma - 1}$
\end{lemma}

\begin{proof}
Choose a $\k$-basis of each cohomology group $H^i(\PP(\mathbf{a}), \F(j))$. This yields a basis of $\T(\F)$ as an $E$-module, and we view the differential on $\T(\F)$ as a matrix with respect to this basis. We fix the following notation: for any $0 \le i \le n+1$, let $w^i$ (resp. $w_i$) denote the maximum (resp. minimum) degree of a square-free monomial in $S$ involving $i$ variables.
We note that $(\om_E)_{(w^{n+1};n+1)} \cong E_{(0;0)} = \k$, and if $d > w^{n+1}$, then $(\om_E)_{(d;i)} = 0$ for all $i \in \Z$. 

Let $\om_E(c; s)$ and $\om_E(d; t)$ be summands in $\T(\F)$ such that the corresponding entry~$f \in E$ of this matrix is nonzero. Since the differential on $\T(\F)$ is minimal of degree $(0;-1)$, the element $f \in E$ has bidegree $(-(c - d) ; -(s-t+1))$, where $0 \le s - t \le n$.  The element $f$ is a sum of squarefree products of $s - t+1$ of the variables in $E$, and so $s - t + 1 \leq c-d\leq w^{s-t + 1}$.  

The terms $\om_E(c;  s)$ and $\om_E(d; t)$ lie in $\Tate(\F)_{c-s}$ and $\Tate(\F)_{d - t}$, respectively. We must show: 
\[
1\leq c -d - (s-t) \leq \sigma + 1.
\]
We have $s - t + 1\leq c-d$, which implies
\begin{equation}
\label{eqn:ineq1}
1\leq c-d-(s-t).
\end{equation}
The inequalities $c-d\leq w^{s-t+1}$ and $n-(s-t) \le w_{n-(s-t)} =w^{n+1} - w^{s-t+1}$ imply
\begin{equation}
\label{eqn:ineq2}
c -d - (s-t) \leq w^{s-t+1} - (s-t) \leq \sigma + 1.
\end{equation}
Combining \eqref{eqn:ineq1} and \eqref{eqn:ineq2} gives the desired inequalities. 
\end{proof}

Given $j \in \Z$, let $\Tate(\F)_{\le j}$ denote the differential submodule $\bigoplus_{\ell \le j} \Tate(\F)_\ell$ of $\Tate(\F)$ arising from the decomposition~\eqref{eqn:filtration}, and set $\Tate(\F)_{> j} \ce \Tate(\F) / \Tate(\F)_{\le j}$.

\begin{lemma}\label{lem:WeightedandR}
Let $M$ be a finitely generated, graded $S$-module and let $\cF=\widetilde{M}$ be the corresponding sheaf on $\PP(\mathbf{a})$. If $r > \reg(M)$, then there is an isomorphism $\RR(M_{\geq r})\xra{\cong} \Tate(\cF)_{\le -r}$ of differential $E$-modules. If~$H^0_\m(M) = 0$, then the same statement holds for $r = \reg(M)$.
\end{lemma}
\begin{proof}
Let $\varepsilon \co F \xra{\simeq} \RR(M_{\ge r})$ be the minimal free resolution of $\RR(M_{\ge r})$. By Theorem~\ref{thm:tate}, we have~$\Tate(\F) \cong \cone(\e) = \RR(M_{\ge r}) \oplus F(0;-1)$. In particular, there is an inclusion $\RR(M_{\ge r}) \into \T(\F)$, and this inclusion clearly takes values in $\T(\F)_{\le -r}$. To show that this inclusion is surjective, we must show that no summand of $F(0; -1)$ lies in $\T(\F)_{\le -r}$. Let $\om_E(-d;i)$ be such a summand. If~$i = 0$, then Corollary~\ref{cor:cohomology} implies that $d < r$, and so $\om_E(-d; i)$ is not contained in $\T(\F)_{\le -r}$. If $i > 0$, then $\om_E(-d;i)$ corresponds to a dimension~1 subspace of $H^i(\PP(\mathbf{a}), \F(j)) \cong H^{i + 1}_\m(M_{\ge r})_j$. Since $M_{\ge r}$ is $r$-regular (Lemma~\ref{rem:trunc}), we must have $-i - j > -r.$ 
\end{proof}

\begin{example}
\label{ex:tatecurve}
Let us return to our running example of a genus 1 curve in $\PP(1, 1, 2)$ (Example~\ref{ex:ellipticCurve}). The weighted regularity of $M = S/(x_0^4+x_1^4+x_2^2)$ is $2$ (Example~\ref{ex:weightedRegEllipticCurve}), and the window of the Tate resolution $\Tate(\OO_C)$ depicted in \eqref{eqn:Tateweighted} illustrates that $\Tate(\OO_C)_{\le -2} = \RR(M_{\ge 2})$, as required by Lemma~\ref{lem:WeightedandR}. Lemma~\ref{lem:jumpbyWn} also tells us that the differential on $\T(\OO_C)$ shifts the filtration in \eqref{eqn:filtration} by either ~1 or $\sigma + 1 = 2$, and \eqref{eqn:Tateweighted} illustrates this as~well. 
\end{example}

\begin{remark}
The size of the window $\Tate(\F)_{\ell - 1} \oplus \cdots \oplus \Tate(\F)_{\ell - \sigma - 1}$ in Lemma~\ref{lem:jumpbyWn} is sharp; that is, we cannot guarantee that the ``jumps" in the Tate resolution differential are any smaller than $\sigma + 1$, in general.
Specifically, consider the Tate resolution of $\cO$ on $\PP(\a)$. For $j>0$, $\cO(-j)$ only has nonzero cohomology when $j\ge\sum_{i=0}^n a_i$, in which case it only has cohomology in degree $n$. If we apply the differential to $\Tate(\cO)_{\sigma}$, Lemma~\ref{lem:jumpbyWn} implies that the image (which must be nonzero by exactness and minimality) is contained in $\Tate(\cO)_{\sigma-1}\oplus \Tate(\cO)_{\sigma-2} \oplus \cdots \oplus \Tate(\cO)_0$. But this sum simply equals $\Tate(\cO)_0$ because the other summands all vanish.  Thus, the statement of Lemma~\ref{lem:jumpbyWn} is sharp in this case.
\end{remark}

\begin{lemma}
\label{lem:weightedFinitePiece}
Let $M$ be a finitely generated graded $S$-module, and let $r > \reg(M)$. 
Denote by $N$ the following $E$-submodule of $\RR(M_{\ge r})$ (recall the notation $\sigma$ from Definition~\ref{defn:symonds}):
\[
N\ce\bigoplus_{d = r}^{r+\sigma+1} \om_E(-d; 0).
\]
Write $\partial|_{N}\colon N\to \RR(M_{\geq r})$ for the restriction of the differential on $\RR(M_{\geq r})$ to $N$. The inclusion of differential $E$-modules $N + \im(\del |_N) \into \RR(M_{\ge r})$ is a quasi-isomorphism. If $H^0_\m(M) = 0$, then this statement also holds when $r = \reg(M)$. 
\end{lemma}

\begin{remark}
The object $N + \im(\del|_N)$ in Lemma~\ref{lem:weightedFinitePiece} is the weighted analogue of the exterior module one resolves to construct the Tate resolution in the standard graded case~\cite[Section 4]{EFS}. The main differences between the standard graded setting and ours are that $N + \im(\del|_N)$ is a \emph{differential} $E$-module rather than an ordinary module, and that the definition of $N + \im(\del|_N)$ involves free summands of $\RR(M_{\ge r})$ generated in $\sigma + 2$ different degrees, rather than just 2 different degrees. 
\end{remark}

\begin{proof}[Proof of Lemma~\ref{lem:weightedFinitePiece}]
Let $\F \ce \widetilde{M}$, and let $\varepsilon \co \Tate(\F)_{> -r} \to \Tate(\F)_{\le -r}$ be the map induced by the differential on $\Tate(\F)$. Since $\Tate(\F)$ is exact, the map $\epsilon$ is a quasi-isomorphism. By Lemma~\ref{lem:WeightedandR}, there is an isomorphism $\Tate(\F)_{\le -r} \cong \RR(M_{\ge r})$ of differential $E$-modules, and so we have a quasi-isomorphism  $\varepsilon \co \Tate(\F)_{> -r} \xra{\simeq} \RR(M_{\ge r})$. By Lemma~\ref{lem:jumpbyWn}, the image of $\varepsilon$ is contained in $N$. The isomorphism induced by~$\varepsilon$ on homology therefore factors as $H(\Tate(\F)_{> -r}) \to H(N + \im(\del|_N)) \into H(\RR(M_{\ge r})$, where the second map is the inclusion. Thus, both maps in this factorization are isomorphisms. 
\end{proof}

\begin{cor}
\label{cor:Ndel}
Adopting the notation of Lemma~\ref{lem:weightedFinitePiece}: if $r > \reg(M)$, then the minimal free resolution of the finitely generated differential $E$-module $N+\im(\partial|_N)$ is isomorphic to $\Tate(\cF)_{>-r}$. If $H^0_\m(M) = 0$, then the same is true when $r = \reg(M)$.
\end{cor}
Combining Corollaries~\ref{cor:cohomology} and~\ref{cor:Ndel}, we arrive at our algorithm: 

\begin{algorithm}\label{alg:mainweighted}
Let $M$ be a finitely generated graded $S$-module, write $\F = \widetilde{M}$, and fix $i, j \in \Z$. To compute~$h^i(\F(j))$, we proceed as follows. 
\begin{enumerate}
\item Compute $\reg(M)$, and let $r >\reg(M)$. When $H^0_\m(M) = 0$, we can also take $r = \reg(M)$.
\item If $j \ge r$, then output $\dim_\kk M_j$ when $i = 0$, and output 0 otherwise. 
\item If $j < r$, then output 0 when $r \le i + j$. Otherwise, continue to Step~(4).
\item Use the \verb|Macaulay2| package \verb|MultigradedBGG| \cite{BBGSTZ} to compute the differential $E$-module $N + \im(\del_{\RR} |_N)$ from Lemma~\ref{lem:weightedFinitePiece}. 
\item Finally, compute a finite approximation of the minimal free flag resolution $\Tate(\F)_{> -r}$ of $N + \im(\del|_N)$ by iterating Algorithm~\ref{alg:resTwisted} $r - i - j$ times, obtaining a free flag differential $E$-module $F$ of the form $F = \bigoplus_{\ell = -r+1}^{-i - j} \Tate(\F)_{\ell}$.
\item  Output $\dim_\k \underline{\Hom}_E(\kk, F)_{(j, -i-1)}$.
\end{enumerate}
\end{algorithm}

\begin{example}
\label{ex:algtate}
Returning to the elliptic curve $C \subseteq \PP(1, 1, 2)$ from Example~\ref{ex:ellipticCurve}, suppose we want to compute $H^i(\mathbb P(1,1,2), \mathcal O_C(j))$ for all $i$ and $-2 \leq j \leq 2$.  Recall that the coordinate ring of $C$ is $M = S/(x_0^4+x_1^4+x_2^2)$, and $\reg(M) = 2$.  We set $N:=\bigoplus_{d=2}^{4} \omega_E(-d;0)\otimes_k M_d$, and we resolve $N + \im(\del|_N)$ as in Lemma~\ref{lem:weightedFinitePiece}.  After three steps of the algorithm, we get the maps from \eqref{eqn:Tateweighted}.  We conclude from this that we have the following:

\begin{center}
\renewcommand{\arraystretch}{1.4}
\begin{tabular}{c|c|c|c|c|c}
   & $j=2$ & $j=1$ & $j=0$ & $j=-1$ & $j=-2$ \\ \hline
$h^0(\cO_C(j))$ & 4 & 2 & 1 & 0 & 0 \\ \hline
$h^1(\cO_C(j))$ & 0 & 0 & 1 & 2 & 4
\end{tabular}
\renewcommand{\arraystretch}{1}
\end{center}
To further illustrate this example, let us demonstrate how to recover it using \verb|Macaulay2|. 
\begin{footnotesize}
\begin{verbatim}

i1 : needsPackage "NormalToricVarieties"; 
i2 : kk = ZZ/101;
i3 : X = weightedProjectiveSpace {1,1,2}; 
i4 : S = ring X; 
i5 : M = S^1 / (x_0^4 + x_1^4 + x_2^2);
i6 : r = weightedRegularity M
o6 = 2
i7 : L = truncatedRR(M, r); 
\end{verbatim}
\end{footnotesize} 
The method \verb|weightedRegularity| computes the regularity $r$ of $M$ as in Definition~\ref{def:reg}. Its implementation is simple, since regularity can easily be obtained from a free resolution using~\eqref{eqn:symonds}. The method \verb|truncatedRR(M,r)| outputs the differential $E$-module $L \ce N + \im(\del|_N)$; this is an application of the method \verb|toricRR| in the package \verb|MultigradedBGG|~\cite{BBGSTZ}. Next, we compute the needed portion $F$ of the Tate resolution of the sheaf~$\OO_C$. Since $r = 2$, the largest value of~$r - i - j$ that can occur for our desired window of the cohomology table ($0 \le i \le 1$ and $-2 \le j \le 2$) is $4$, so we need to iterate Algorithm~\ref{alg:resTwisted} at most~four times. But we know $h^0(\OO_C(-2)) = 0$, so in fact we only need three iterations. 
\begin{footnotesize}
\begin{verbatim}

i8 : F = twistedFlagRes(L,3);
i9 : F.dd_0
o9 = {5, 2} | 0 0 0 e_0 0   0   0      0      0   e_1e_2 |
     {5, 2} | 0 0 0 e_1 0   0   0      e_0e_2 0   0      |
     {4, 1} | 0 0 0 0   e_1 e_0 -e_2   0      0   0      |
     {4, 2} | 0 0 0 0   0   0   e_0e_1 0      0   0      |
     {3, 1} | 0 0 0 0   0   0   0      -e_1   e_0 0      |
     {3, 1} | 0 0 0 0   0   0   0      0      e_1 e_0    |
     {2, 1} | 0 0 0 0   0   0   0      0      0   0      |
     {2, 1} | 0 0 0 0   0   0   0      0      0   0      |
     {2, 1} | 0 0 0 0   0   0   0      0      0   0      |
     {2, 1} | 0 0 0 0   0   0   0      0      0   0      |

\end{verbatim}
\end{footnotesize}
The method \verb|twistedFlagRes(L,3)| executes Algorithm~\ref{alg:resTwisted} three times. The output is an object of type \verb|DifferentialModule|, which is encoded as a \verb|Complex| with three identical nonzero terms in homological degrees $-1, 0$, and 1 and two identical differentials. To compute our portion of the cohomology table, we only need the degrees of the generators of $F$, which are the bidegrees on the left of this matrix. The summands of $F$ correspond to those appearing in~\eqref{eqn:Tateweighted}. To explain the correspondence, recall first that $\om_E$ is generated in degree $(4 ; 3)$, and second that~$\T(\OO_C) = F(0; -1) \oplus \RR(M_{\ge 2})$, so that the generating degrees of the summands in~\eqref{eqn:Tateweighted} (except for $\om_E(-2;0)^{\oplus 4} \subseteq \RR(M_{\ge 2})$) coincide with the degrees appearing in the \verb|Macaulay2| output above, up to twisting by $(0 ; -1)$. 

\medskip
\noindent Finally, with this data in hand, we read off the cohomology table of $\OO_C$ in the desired range:
\begin{footnotesize}
\begin{verbatim}

i9 : weightedCohomologyTable(M, {0, 1}, {-2, 2})
+---+-----------+-----------+--------+------------+------------+
| i | h^i(F(2)) | h^i(F(1)) | h^i(F) | h^i(F(-1)) | h^i(F(-2)) |
+---+-----------+-----------+--------+------------+------------+
| 0 |         4 |         2 |      1 |          0 |          0 |
| 1 |         0 |         0 |      1 |          2 |          4 |
+---+-----------+-----------+--------+------------+------------+
\end{verbatim}
\end{footnotesize}
The inputs $\{0, 1\}$ and $\{-2, 2\}$ encode our desired window of the cohomology table. We used four methods above that are not yet distributed in \verb|Macaulay2|: \verb|weightedRegularity|, \verb|truncatedRR|, \verb|twistedFlagRes|, and \verb|weightedCohomologyTable|. As discussed in the introduction, an efficient implementation of this algorithm will rely on fast methods for homological computations with differential modules that have not yet been fully developed.
\end{example}

\begin{example}
Let us consider a weighted rational curve, such as those studied in \cite{davis-sobieska,BEsyzygies}.  Specifically, we consider a curve $C$ of type $(d,e) = (3,2)$ as in \cite[Example 1.1]{davis-sobieska}.  This is the image of $\PP^1 \to \PP(1^3,2^2)$ where $[s:t]\mapsto [s^3:s^2t:st^2:st^5:t^6]$.  The coordinate ring of this curve is~$S/I$, where $I$ is generated by the $2 \times 2$ minors of the matrix
$$
\begin{pmatrix}
x_0 & x_1 & x_2^2 & x_3 \\
x_1 & x_2 & x_3 & x_4
\end{pmatrix}
$$
We have $\reg(S/I) = 1$.  Applying Algorithm~\ref{alg:mainweighted} yields the following piece of the Tate resolution:
\begin{equation}
\xymatrix{
\cdots \ar[r] & \om_E(2;1)^{\oplus 5}\ar[r]\ar@/_-2pc/[rr]^-{}&\om_E(1;1)^{\oplus 3}\ar@/_2pc/[rr]^-{}\ar[r] & \om_E(0;0)\ar@/_-2pc/[rr]^-{} \ar[r] &  \omega_E(-1;0)^{\oplus 3}\ar[r] \ar@/_2pc/[rr]^-{}&\omega_E(-2;0)^{\oplus 7} \ar[r]& \cdots
}
\end{equation}
We thus compute the following portion of the cohomology table of $\OO_C$:
\begin{center}
\renewcommand{\arraystretch}{1.4}
\begin{tabular}{c|c|c|c|c|c}
   & $j=2$ & $j=1$ & $j=0$ & $j=-1$ & $j=-2$ \\ \hline
$h^0(\cO_C(j))$ & 7 & 3 & 1 & 0 & 0 \\ \hline
$h^1(\cO_C(j))$ & 0 & 0 & 0 & 3 & 5
\end{tabular}
\renewcommand{\arraystretch}{1}
\end{center}

\end{example}


\section{Adapting the algorithm to weighted projective spaces}\label{sec:variety}
Algorithm~\ref{alg:mainweighted} computes sheaf cohomology over weighted projective stacks.  With a minor variant of the algorithm, we now show how to compute sheaf cohomology tables on weighted projective {\em spaces}. More specifically: given a coherent sheaf $\F$ on $\PPv(\a)$, we use Tate resolutions to give an algorithm for computing the cohomology groups of $\F(j) \ce \F \otimes \OO(j)$ for all $j$ in some specified finite range. The key extra winkle, which was highlighted in \S\ref{rmks:varietyVsStack} and Example~\ref{ex:stackVsVariety}, is that $\widetilde{M}(j)$ may not be equal to $\widetilde{M(j)}$.  That is:  twisting by $j$ may fail to commute with the ``tilde'' operation of passing from a graded module to the corresponding sheaf.

\medskip
We first show that Tate resolutions compute $H^i(\PPv(\a), \F \otimes \OO(j))$ when $\OO(j)$ is a line bundle, i.e. when $j$ is a multiple of $\ell \ce \on{lcm}(a_0, \dots, a_n)$.

\begin{lemma}\label{lem:lcmAndCoarse}
Let $\pi\colon \PP(\mathbf{a})\to \PPv(\mathbf{a})$ be the coarse moduli space map. 
    Let $M$ be a finitely generated graded $S$-module and $\cF_v$ (resp. $\cF_s$) the corresponding sheaf on $\PPv(\mathbf{a})$ (resp.~$\PP(\mathbf{a})$). We have: 
    \begin{enumerate}
    \item $\pi_*(\F_s) \cong \F_v$.
    \item $\pi^*(\OO_{\PPv(\a)}(e\ell)) \cong \OO_{\PP(\a)}(e\ell)$ for all $e \in \Z$.
    \item $\pi_*(\F_s(e\ell)) \cong \F_v(e\ell)$ for all $e \in \Z$.
    \item $H^i(\PPv(\mathbf{a}), \cF_v(e\ell)) \cong H^i(\PP(\mathbf{a}), \cF_s(e\ell)) \cong \Hom_E(\kk, \Tate(\cF_s))_{(e\ell; -i)}$ for all $e,i \in \Z$.
    \end{enumerate}
\end{lemma}

\begin{proof}
Let $F \co \Db(\PP(\a)) \to \Db(\PPv(\a))$ denote the triangulated functor on bounded derived categories given by sending an object $\mathcal{C} \in \Db(\PP(\a))$ to the complex of sheaves on $\PPv(\a)$ associated to the complex $\bigoplus_{i \ge 0} \RR\Gamma(\PP(\a), \cC(i))$ of graded $S$-modules. Observe that $F(\F_s) = \F_v$. By \cite[Proposition 2.23]{king}, we have $F(\OO(j)) \cong \pi_*(\OO(j))$ when~$-a < j \le 0$. We also have $\pi_* =~\RR\pi_*$~\cite[Proposition 3.5]{EM12}, and so $\pi_*$ induces a triangulated functor $\Db(\PP(\a)) \to~\Db(\PPv(\a))$. Since the objects $\OO(-a + 1), \dots, \OO$ generate $\Db(\PP(\a))$, we conclude that $F = \pi_*$, which implies (1). 

Let $e \in \Z$. Since $\OO_{\PPv(\a)}(e\ell)$ is a line bundle, we have $\pi^*(\OO_{\PPv(\a)}(e\ell)) \cong \OO_{\PP(\a)}(i)$ for some $i \in \Z$. Applying (1), we therefore have $\pi^*\pi_*(\OO_{\PP(\a)}(e\ell)) \cong \OO_{\PP(\a)}(i)$. Since $\pi_*\pi^*\pi_*$ is isomorphic to $\pi_*$, applying~$\pi_*$ to this isomorphism and using (1) again implies $\OO_{\PPv(\a)}(e\ell) \cong \OO_{\PPv(\a)}(i)$, i.e. $i = e\ell$. This proves~(2), and (3) then follows from combining (1) and (2) with the projection formula. Part~(4) follows from~(3), since~$\pi_*$ preserves cohomology~\cite[Proposition 3.5]{EM12}. 
\end{proof}

It follows that, given a finitely generated graded $S$-module $M$ on $\PPv(\a)$, Algorithm~\ref{alg:mainweighted} computes $H^*(\PPv(\a), \widetilde{M}(j))$ provided that $j$ is a multiple of $\ell$. 
Thus, to compute $H^*(\PPv(\a), \widetilde{M}(j))$ for \emph{all} $j$, we can do the following:
\begin{enumerate}
    \item   For each $0\leq j <\ell$, find a module $M'$ such that the sheaf $\widetilde{M'}$ on $\PPv(\mathbf{a})$ is isomorphic to the sheaf $\widetilde{M}(j)$.
    \item  Compute the Tate resolution $\Tate(\widetilde{M'})$ associated to $M'$.
    \item  ``Extract'' from $\Tate(\widetilde{M'})$ the sheaf cohomology groups corresponding to twists of $\widetilde{M'}$ by multiples of $\ell$, which in turn correspond to twists of $\widetilde{M}$ by degrees that are congruent to $j$ modulo $\ell$.
\end{enumerate}  
This raises a key issue: as discussed in \S\ref{rmks:varietyVsStack}, for $0<j<\ell$, it might be the case that $\widetilde{M}(j) \ne \widetilde{M(j)}$.  So how do we achieve Step (1)?  {\em How do we find a module $M'$ where $\widetilde{M'} \cong \widetilde{M}(j)$}?  This is a subtle point, and it does not appear to have been answered either in computational programs like~\cite{M2}, or in the computational literature like~\cite{EMS}. We answer this question in Proposition~\ref{prop:tensorProduct}. Before we state and prove it, we require some additional setup. 



\begin{lemma}\label{lem:tensorFreeMod}
    Let $F$ be a graded free $S$-module that is generated in degrees that are multiples of $\ell$.  For any  graded $S$-module $M$, we have
    \[
    \widetilde{M\otimes F} \cong \widetilde{M}\otimes \widetilde{F}
    \]
    as coherent sheaves on $\PPv(\mathbf{a})$.
\end{lemma}

\begin{proof}
When $M$ is free, this follows from \cite[Corollary 4A.5(b)]{BR}. In general, choose a presentation of $M$, and use that tensor products are right exact and the functor $\widetilde{\phantom{M}}$ is~exact. 
\end{proof}



\begin{notation}\label{notation:Phij}
    For each $1\leq j \leq \ell-1$, we write $S(j)^{(\ell)} \ce \bigoplus_{e\in \ZZ} S_{j+e\ell}$, i.e. the degree $\ell$ Veronese submodule of $S(j)$.  Let $\phi_j\colon G'_j\to F'_j$ be a presentation matrix for $S(j)^{(\ell)}$ over the  ring $S^{(\ell)}$, so that $G'_j$ and $F'_j$ are graded, free $S^{(\ell)}$-modules.
Let $\Phi_j: G_j\to F_j$ be the corresponding matrix over the ring $S$, where $G_j$ and $F_j$ are the corresponding free $S$-modules.  In other words: choosing bases of $F'_j$ and $G'_j$, we can represent $\phi_j$ as an $m \times n$ matrix $\bigoplus_{t = 1}^n S^{(\ell)}(a_t) \to  \bigoplus_{s = 1}^m S^{(\ell)}(b_s)$ for some twists $a_1, \dots, a_n, b_1, \dots, b_m \in \Z$. The map~$\Phi_j$ is the same matrix, considered as a map $\bigoplus_{t = 1}^n S^{}(a_t \ell) \to \bigoplus_{s = 1}^m S(b_s\ell)$.
\end{notation} 
\begin{lemma}
\label{lem:sheafify}
Using Notation~\ref{notation:Phij}, we have $\widetilde{\coker(\Phi_j)} \cong \cO_{\PPv(\mathbf{a})}(j)$.
\end{lemma}

\begin{proof}
Let $p_s \in S(j)^{(\ell)}$ be the image of $1 \in S^{(\ell)}(b_s)$ under the map $F_j' \to S(j)^{(\ell)}$ for $1 \le s \le m$. The map $F_j \to S(j)$ that sends $1 \in S(b_s\ell)$ to $p_s$ induces a map $\coker(\Phi_j) \to S(j)$ that is an isomorphism in degrees  $e\ell$ for $e \in \Z$. The result thus follows from \cite[Exercise 5.3.5]{CLS}.
\end{proof}

\begin{prop}\label{prop:tensorProduct}
    Let $M$ be a graded $S$-module, and let $\widetilde{M}$ be the corresponding  sheaf on $\PPv(\mathbf{a})$. 
    Let $j\in \ZZ$, and let $\Phi_j, F_j, G_j$ be as above.  
    Using Notation~\ref{notation:Phij}, set
    $$
    M' \ce \coker(M\otimes \Phi_j) = \coker(M\otimes F_j \overset{\Phi_j}{\gets} M \otimes G_j).
    $$
    There is an isomorphism $\widetilde{M'} \cong \widetilde{M}(j)$ of  sheaves on $\PPv(\mathbf{a})$.
\end{prop}
\begin{proof}
We have:
\begin{align*}
\widetilde{M}\otimes \cO_{\PPv(\mathbf{a})}(j) &= \widetilde{M} \otimes \coker\left(\widetilde{F_j} \overset{\Phi_j}{\longleftarrow}\widetilde{G_j}\right)&\text{Lemma~\ref{lem:sheafify} and exactness of $\widetilde{\phantom{M}}$.}\\
&= \coker\left(\widetilde{M} \otimes \widetilde{F_j} \overset{\Phi_j}{\longleftarrow} \widetilde{M}\otimes \widetilde{G_j}\right)&\text{Right exactness of $\otimes$.}\\
&= \coker\left(\widetilde{M\otimes F_j} \overset{\Phi_j}{\longleftarrow} \widetilde{M\otimes G_j}\right)&\text{Lemma~\ref{lem:tensorFreeMod}.}\\
&= \widetilde{M'}&\text{Exactness of $\widetilde{\phantom{M}}$.}
\end{align*}
\end{proof}

We now present our algorithm for computing sheaf cohomology on the weighted projective variety.

\begin{algorithm}\label{alg:variety}
    Let $M$ be a finitely generated graded $S$-module, and write $\mathcal F = \widetilde{M}$ for the corresponding sheaf on the weighted projective variety $\PPv(\mathbf{a})$. Fix $i, j \in \Z$. To compute $H^i(\PPv(\mathbf{a}), \mathcal F(j))$ we proceed as follows.
\begin{enumerate}
    \item Let $j_0$ be the unique integer $0\leq j_0<\ell$ that is congruent to $j$ modulo $\ell$.   Use the {\tt Macaulay2} package {\tt PushForward} to compute the matrix $\Phi_{j_0}$ as in Notation~\ref{notation:Phij}, and let $M' \ce \coker(M\otimes \Phi_{j_0})$, as in Proposition~\ref{prop:tensorProduct}.
     \item Compute $\reg(M')$, and let $\rho>\reg(M')$. (When $H^0_{\mathfrak m}(M')=0$, we can choose $\rho \geq \reg(M')$.)
    \item  If $j - j_0 \geq  \rho$, then output $\dim_\kk(M')_{j - j_0}$ when $i=0$, and output $0$ otherwise.
    \item  If $j - j_0< \rho$, then output 0 when $\rho \le i + j - j_0$. Otherwise, continue to Step~(5).
    \item  Use the {\tt Macaulay2} package {\tt MultigradedBGG} \cite{BBGSTZ} to compute the differential $E$-module $N + \im(\del_{\RR} |_N)$ from Lemma~\ref{lem:weightedFinitePiece}, but applied to the module $M'$ in place of $M$.
    \item  Compute a finite approximation of the minimal free flag resolution $\Tate(\F(j_0))_{> -\rho}$ of the differential $E$-module $N + \im(\del|_N)$ by iterating Algorithm~\ref{alg:resTwisted}  $\rho - i - j + j_0$ times, obtaining a free flag differential $E$-module $F$ of the form $F = \bigoplus_{s = -\rho+1}^{-i - j + j_0} \Tate(\F(j_0))_s$.
    \item  Output $\dim_\k \underline{\Hom}_E(\kk, F)_{(j-j_0, -i-1)}$.
\end{enumerate}
\end{algorithm}
The following example indicates how we can use Algorithm~\ref{alg:variety} to compute the cohomology of a sheaf on $\cF$ on $\PPv(\mathbf{a})$ with respect to some window of degrees $[j_0,j_1]\subseteq \ZZ$.  

\begin{example}\label{ex:backToStackyPoint}
    Let us return to Example~\ref{ex:stackVsVariety}.  Recall that $q$ is the singular point $[0:0:1]$ on $\PPv(1,1,2)$.  We write $M=S/(x_0,x_1)$ for the module such that $\widetilde{M}=\cO_q$. Our goal is to compute the sheaf cohomology $\widetilde{M}$ with respect to all twists in $\ZZ$.

    For the even twists, we compute the Tate resolution of $M$ as in Example~\ref{ex:stackVsVariety}, yielding:
\begin{equation}\label{eqn:Tateeven}
\Tate(\widetilde{M})= \left[
\xymatrixcolsep{5mm}
\xymatrix{
\cdots \ar@/_1pc/[rr]^-{}& 0 & \om_E(2;0) \ar@/_-1pc/[rr]^-{} &0 & \om_E(0;0) \ar@/_1pc/[rr]^-{} &0& \omega_E(-2;0)\ar@/_-1pc/[rr]^-{}&0& \cdots
}
\right].
\end{equation}
By Algorithm~\ref{alg:variety}, this Tate resolution will give the correct sheaf cohomology groups of $\mathcal O_q(j)$ for all twists by even degrees $j$, i.e. the nonzero summands depicted above.  In particular, we see that $H^0(\PPv(1,1,2), \cO_q(2n)) = \kk$ for all $n\in \ZZ$.

For the odd degrees, we need to first apply Proposition~\ref{prop:tensorProduct}.  We have the presentation:
\[
\cO_{\PPv(1,1,2)}(1) \gets \cO_{\PPv(1,1,2)}^2 \overset{\Phi_1}{\gets} \cO_{\PPv(1,1,2)}^2(-2), \quad \text{where} \quad \Phi_1 = \begin{pmatrix}
    x_0^2& -x_0x_1\\ -x_0x_1 & x_1^2
\end{pmatrix}.
\] 
We set $F_1=S^2$ and $G_1=S(-2)^2$, and we view $\Phi_1$ as a map $G_1\to F_1$ as above.  By Proposition~\ref{prop:tensorProduct}, the sheaf on $\PPv(1,1,2)$ associated to $M'=\coker(M\otimes \Phi_1)$ is isomorphic to $\widetilde{M}(1)$.  The Tate resolution of $M'$ is therefore
\begin{equation}\label{eqn:Tateodd}
\Tate(\widetilde{M'})= \left[
\xymatrixcolsep{5mm}
\xymatrix{
\cdots \ar@/_1pc/[rr]^-{}& 0 & \om_E(2;0)^2 \ar@/_-1pc/[rr]^-{} &0 & \om_E(0;0)^2 \ar@/_1pc/[rr]^-{} &0& \omega_E(-2;0)^2\ar@/_-1pc/[rr]^-{}&0& \cdots
}\right],
\end{equation}
and this gives the correct sheaf cohomology groups of $\mathcal O_q(j)$ for all twists by odd degrees $j$, which once again correspond to the nonzero summands above.

We can thus weave together a cohomology table for $\cO_q$ with respect to all twists using \eqref{eqn:Tateeven} and~\eqref{eqn:Tateodd}.  The terms $\omega_E(2e;0)$ from \eqref{eqn:Tateeven} corresponds to the even twists  $H^0(\PPv(1,1,2),\cO_q(2e))$, and the terms $\omega_E(2e;0)^2$ from \eqref{eqn:Tateodd} correspond to the odd twists $H^0(\PPv(1,1,2), \cO_q(2e+1))$. We thus conclude that
\[
H^i(\PPv(1,1,2), \cO_q(j)) = \begin{cases}
    0,& i\ne 0;\\
    \kk, & i=0 \text{ and } j \text{ even};\\
    \kk^2, & i=0 \text{ and } j \text{ odd}.\\
\end{cases}.
\]
\end{example}

\begin{remark}\label{rmk:complexity}
The complexity of Algorithm~\ref{alg:variety} is roughly equivalent to that of Algorithm~\ref{alg:mainweighted}. Specifically: computing the cohomology of $\cF(j)$ on $\PPv(\mathbf{a})$ for $j$ in some window of degrees $[j_0,j_1]\subseteq~\ZZ$ amounts to computing the Tate resolution of $\cF, \cF(1), \cF(2), \dots, \cF(\ell-1)$ and reading off the appropriate Betti numbers from each.  In summary, the added complexity of computing cohomology over the variety versus over the stack is essentially linear, and controlled by the lcm $\ell$.    
\end{remark}

\bibliographystyle{amsalpha}
\bibliography{Bibliography}

\end{document}